\documentclass[11 pt]{article}
\title{The abelianization of the level $L$ mapping class group}
\author{Andrew Putman}
\date{March 12, 2008}
\usepackage{amsmath}
\usepackage{amssymb}
\usepackage{amsthm}
\usepackage{epsfig}
\usepackage{geometry}
\usepackage{amsfonts}
\usepackage[font=small]{caption}
\usepackage{amscd}
\usepackage{pinlabel}

\theoremstyle{plain}
\newtheorem{theorem}{Theorem}[section]

\newtheorem{lemma}[theorem]{Lemma}

\newcommand\BeginClaimProof{\begin{proof}[{Proof of Claim \theclaim}]}
\newcommand\EndClaimProof{\end{proof}}

\theoremstyle{definition}
\newtheorem{definition}[theorem]{Definition}

\theoremstyle{remark}
\newtheorem*{remark}{Remark}

\newtheorem*{acknowledgements}{Acknowledgments}

\DeclareMathOperator{\Diff}{Diff}

\DeclareMathOperator{\Ker}{ker}

\DeclareMathOperator{\Mod}{Mod}
\newcommand\Torelli{\text{${\mathcal I}$}}
\DeclareMathOperator{\Sp}{Sp}

\DeclareMathOperator{\SL}{SL}

\newcommand\Z{\text{$\mathbb{Z}$}}
\newcommand\Q{\text{$\mathbb{Q}$}}

\DeclareMathOperator{\HH}{H}

\newcommand\Figure[3]{
\begin{figure}[t]
\centering
\centerline{\psfig{file=#2,scale=60}}
\caption{#3}
\label{#1}
\end{figure}}

\newcommand\Ident{\text{$\mathbb{I}$}}
\newcommand\Zero{\text{$\mathbb{O}$}}
\newcommand\MatTwoTwo[4]{\text{$\left( \begin{smallmatrix} #1 & #2\\ #3 & #4 \end{smallmatrix} \right)$}}
\newcommand\ElemMat{\text{$\mathcal{E}$}}
\newcommand\SymMat{\text{$\mathcal{SE}$}}
\newcommand\SpGenX{\text{$\mathcal{X}$}}
\newcommand\SpGenY{\text{$\mathcal{Y}$}}
\newcommand\SpGenZ{\text{$\mathcal{Z}$}}
\newcommand\SpLie{\text{$\mathfrak{sp}$}}
\newcommand\SLLie{\text{$\mathfrak{sl}$}}

\begin{document}

\maketitle

\begin{abstract}
We calculate the abelianizations of the level $L$ subgroup of the genus $g$ mapping class group
and the level $L$ congruence subgroup of the $2g \times 2g$ symplectic group
for $L$ odd and $g \geq 3$.
\end{abstract}

\noindent
{\bf Historical note.}  I originally wrote this paper in March of 2008.  Towards the end of that month,
I gave a master class on the Torelli group at the University of Aarhus.  That master class ended in
a conference, and I had intended to speak about this paper at that conference.  However, I learned
that both Bernard Perron and Masatoshi Sato had proven similar theorems and intended to speak about
them at the same conference!  Sato was a graduate student and had actually proved somewhat better
results (in particular, he could deal with $L=2$), so I decided not to publish this paper.  
Sato's work appeared in \cite{SatoAbel}, and
Perron's work was sketched in \cite{PerronAbel}.  See my later paper \cite{PutmanPicardGroupLevel} for
results for $L$ not divisible by $4$.  Dealing with the case where $L$ is divisible by $4$ is still open.

\section{Introduction}

Let $\Sigma_{g,n}$ be an orientable genus $g$ surface with $n$ boundary components and let
$\Mod_{g,n}$ be its {\em mapping class group}, that is, the group $\pi_0(\Diff^{+}(\Sigma_{g,n},\partial \Sigma_{g,n}))$.
This is the (orbifold) fundamental group of the moduli space of Riemann surfaces and has been
intensely studied by many authors.  For $n \in \{0,1\}$, the action of $\Mod_{g,n}$ on $\HH_1(\Sigma_{g,n};\Z)$
induces a surjective representation of $\Mod_{g,n}$ into the symplectic group whose kernel $\Torelli_{g,n}$
is known as the {\em Torelli group}.  This is summarized by the exact sequence
$$1 \longrightarrow \Torelli_{g,n} \longrightarrow \Mod_{g,n} \longrightarrow \Sp_{2g}(\Z) \longrightarrow 1.$$

For $L \geq 2$, let $\Sp_{2g}(\Z,L)$ denote the {\em level $L$ congruence subgroup} of $\Sp_{2g}(\Z)$, that is,
the subgroup of matrices that are equal to the identity modulo $L$.  The pull-back of $\Sp_{2g}(\Z,L)$ to
$\Mod_{g,n}$ is known as the {\em level $L$ subgroup} of $\Mod_{g,n}$ and is denoted by $\Mod_{g,n}(L)$.  The
group $\Mod_{g,n}(L)$ can also be described as the group of mapping classes that act trivially on 
$\HH_1(\Sigma_{g,n};\Z / L \Z)$.  It fits into an exact sequence
$$1 \longrightarrow \Torelli_{g,n} \longrightarrow \Mod_{g,n}(L) \longrightarrow \Sp_{2g}(\Z,L) \longrightarrow 1.$$
In \cite{HainTorelli}, Hain proved that the abelianization of $\Mod_{g,n}(L)$ consists entirely of torsion for $g \geq 3$ (an
alternate proof was given by McCarthy in \cite{McCarthyLevel}).
In this note, we compute this torsion for $L$ odd.

To state our theorem, we need some notation.  Denoting the $n \times n$ zero matrix by $\Zero_n$ and the $n \times n$
identity matrix by $\Ident_n$, let $\Omega_g$ 
be the matrix $\MatTwoTwo{\Zero_g}{\Ident_g}{-\Ident_g}{\Zero_g}$ (we will abuse notation and let
the entries of $\Omega_g$ lie in whatever ring we are considering at the moment).
By definition, the group $\Sp_{2g}(\Z)$ consists of $2g \times 2g$ integral matrices $X$ that satisfy 
$X^t \Omega_{g} X = \Omega_g$.  We will denote by $\SpLie_{2g}(L)$ the additive group of all $2g \times 2g$ matrices 
$A$ with entries in $\Z / L \Z$ that satisfy $A^t \Omega_{g} + \Omega_g A = 0$.

Our main theorem is as follows, and is proven in \S \ref{section:mainh1level}.

\begin{theorem}[{Integral $\HH_1$ of level $L$ subgroups}]
\label{theorem:mainh1}
For $g \geq 3$, $n \in \{0,1\}$, and $L$ odd, set $H(L) = \HH_1(\Sigma_{g,n};\Z / L \Z)$.  We then
have an exact sequence
$$0 \longrightarrow K \longrightarrow \HH_1(\Mod_{g,n}(L);\Z) \longrightarrow \SpLie_{2g}(L) \longrightarrow 0,$$
where $K = \wedge^3 H(L)$ if $n=1$ and $K = (\wedge^3 H(L))/H(L)$ if $n=0$.  
\end{theorem}

\begin{remark}
The condition $g \geq 3$ is necessary, since in \cite{McCarthyLevel} McCarthy proves that if $2$ or $3$ divides
$L$, then $\Mod_{2}(L)$ surjects onto $\Z$.  A computation of $\HH_1(\Mod_{2,n}(L);\Z)$ (or even
$\HH_1(\Mod_{2,n}(L);\Q)$) would be very interesting.
\end{remark}

We now describe the sources for the terms in the exact sequence of Theorem \ref{theorem:mainh1}.
The kernel $K$ comes from the {\em relative Johnson homomorphisms} of Broaddus-Farb-Putman 
\cite{BroaddusFarbPutmanDistortion}.  For $\Mod_{g,n}(L)$, these are surjective homomorphisms
$$\tau_{g,1}(L) : \Mod_{g,1}(L) \longrightarrow \wedge^3 H(L)$$
and
$$\tau_g(L) : \Mod_{g}(L) \longrightarrow (\wedge^3 H(L)) / H(L)$$
which are related to the celebrated Johnson homomorphisms on the Torelli group (see 
\S \ref{section:torelli} and \S \ref{section:mainh1level}).

The cokernel $\SpLie_{2g}(L)$ is the abelianization of $\Sp_{2g}(\Z,L)$.  Now, 
the isomorphism 
$$\HH_1(\Sp_{2g}(\Z,L);\Z) \cong \SpLie_{2g}(L)$$ 
can be deduced from general theorems of Borel on arithmetic groups (see \cite[\S 2.5]{BorelAbel}); however, Borel's
results are much more general than we need and it takes some work to derive
the desired result from them.  We instead imitate a beautiful argument of Lee-Szczarba \cite{LeeSzCong},
who prove that
$$\HH_1(\SL_{n}(\Z,L);\Z) \cong \SLLie_{n}(L)$$
for $n \geq 3$.  Here $\SL_{n}(\Z,L)$ is the level $L$ congruence subgroup of $\SL_{n}(\Z)$ and $\SLLie_{n}(L)$ is
the additive group of $n \times n$ matrices with coefficients in $\Z / L \Z$ and trace $0$.  The
proof of the following theorem is contained in \S \ref{section:mainh1sp}.

\begin{theorem}[{Integral $\HH_1$ of $\Sp_{2g}(\Z,L)$}]
\label{theorem:h1spl}
For $g \geq 3$ and $L$ odd, we have
$$\HH_1(\Sp_{2g}(\Z,L);\Z) \cong \SpLie_{2g}(L).$$
Moreover, $[\Sp_{2g}(\Z,L),\Sp_{2g}(\Z,L)] = \Sp_{2g}(\Z,L^2)$.
\end{theorem}

\begin{remark}
It is unclear whether the hypothesis that $L$ is odd is necessary for Theorems \ref{theorem:mainh1} or \ref{theorem:h1spl},
but it is definitely used in both proofs.
\end{remark}

\begin{acknowledgements}
I wish to thank Nate Broaddus and Benson Farb, as portions of this paper came out of conversations
arising from our joint work \cite{BroaddusFarbPutmanDistortion}.  I also wish to thank Tom Church
for several useful comments and suggestions.
\end{acknowledgements}

\section{The abelianization of $\boldsymbol{\Sp_{2g}(\Z,L)}$}
\label{section:mainh1sp}

We will need the following notation.

\begin{definition}
For $1 \leq i,j \leq n$, let $\ElemMat_{i,j}^n(r)$ be the $n \times n$ matrix
with an $r$ at position $(i,j)$ and $0$'s elsewhere.  Similarly, let $\SymMat_{i,j}^n(r)$ be the
$n \times n$ matrix with an $r$ at positions $(i,j)$ and $(j,i)$ and $0$'s elsewhere.
\end{definition}

\begin{definition}
For $1 \leq i,j \leq g$, denote by $\SpGenX_{i,j}^g(r)$ the matrix
$\MatTwoTwo{\Ident_g}{\Zero_g}{\SymMat_{i,j}^g(r)}{\Ident_g}$, by $\SpGenY_{i,j}^g(r)$ the matrix
$\MatTwoTwo{\Ident_g}{\SymMat_{i,j}^g(r)}{\Zero_g}{\Ident_g}$,
and by $\SpGenZ_{i,j}^g(r)$ the matrix 
$\MatTwoTwo{\Ident_g+\ElemMat_{i,j}^g(r)}{\Zero_g}{\Zero_g}{\Ident_g - \ElemMat_{j,i}^g(r)}$.
\end{definition}

Observe that $\SpGenX_{i,j}^g(L),\SpGenY_{i,j}^g(L) \in \Sp_{2g}(\Z,L)$ for all $1 \leq i,j \leq g$ and
that $\SpGenZ_{i,j}^g(L) \in \Sp_{2g}(\Z,L)$ for $1 \leq i,j \leq g$ with $i \neq j$.
The following theorem forms part of Bass-Milnor-Serre's solution to the congruence
subgroup problem for the symplectic group.

\begin{theorem}[{Bass-Milnor-Serre \cite[Theorem 12.4, Corollary 12.5]{BassMilnorSerre}}]
\label{theorem:splgen}
For $g \geq 2$ and $L \geq 1$, the group $\Sp_{2g}(\Z,L)$ is normally generated by
$$\{\text{$\SpGenX_{i,j}^g(L)$ $|$ $1 \leq i,j \leq g$}\} \cup \{\text{$\SpGenY_{i,j}^g(L)$ $|$ $1 \leq i,j \leq g$}\}.$$
\end{theorem}

\begin{remark}
We emphasize that the matrices $\SpGenZ_{i,j}^g(L)$ are not needed -- the proof of \cite[Lemma 13.1]{BassMilnorSerre} contains
an explicit formula for them in terms of the $\SpGenX_{i,j}^g$ and the $\SpGenY_{i,j}^g$.
\end{remark}

\noindent
Using this, we can prove the following.

\begin{lemma}
\label{lemma:spcommutatorcalc}
For $g \geq 3$ and $L$ odd, we have $\Sp_{2g}(\Z,L^2) < [\Sp_{2g}(\Z,L),\Sp_{2g}(\Z,L)]$.
\end{lemma}
\begin{proof}
We must show that each normal generator of $\Sp_{2g}(\Z,L^2)$ given by Theorem \ref{theorem:splgen}
is contained in $[\Sp_{2g}(\Z,L),\Sp_{2g}(\Z,L)]$.  We will do the case of $\SpGenX_{i,j}^g(L^2)$; the
other case is similar.  Assume first that $i \neq j$. Since $g \geq 3$, there is some $1 \leq k \leq g$
so that $k \neq i,j$.  The following matrix identity then proves the desired claim:
$$\SpGenX_{i,j}^g(L^2) = [\SpGenX_{i,k}^g(L), \SpGenZ_{k,j}^g(L)].$$
Now assume that $i = j$.  Again, there exists some $1 \leq k_1<k_2 \leq g$ so that $k_1,k_2 \neq i$.  Also,
since $L$ is odd there exists some integer $N$ so that $2N+L = 1$.  We thus have
$$\SpGenX_{i,i}^g(L^2) = \SpGenX_{i,i}^g((2N+L)L^2) = \SpGenX_{i,i}^g(2N L^2) \cdot \SpGenX_{i,i}^g(L^3),$$
so the following matrix identities complete the proof:
\begin{align*}
\SpGenX_{i,i}^g(2N L^2) &= [\SpGenX_{i,k_1}^g(N L), \SpGenZ_{k_1,i}^g(L)],\\
\SpGenX_{i,i}^g(L^3)    &= [\SpGenX_{k_1,k_1}^g(L), \SpGenZ_{k_1,i}^g(L)] \cdot [\SpGenZ_{k_2,i}^g(L), \SpGenX_{k_1,k_2}^g(L)]. \qedhere
\end{align*}
\end{proof}

\begin{proof}[Proof of Theorem \ref{theorem:h1spl}]
We begin by defining a function $\phi : \Sp_{2g}(\Z,L) \rightarrow \SpLie_{2g}(L)$.  Consider
any matrix $X \in \Sp_{2g}(\Z,L)$.  Write $X = \Ident_{2g} + L A$, and define
$$\phi(X) = A \quad \pmod{L}.$$
We claim that $\phi(X) \in \SpLie_{2g}(L)$.  Indeed, by the definition of the symplectic group
we have $X^t \Omega_{g} X = \Omega_g$.  Writing $X = \Ident_{2g} + L A$ and expanding out, we have
$$\Omega_g + L (A^t \Omega_{g} + \Omega_{g} A) + L^2 (A^t \Omega_{g} A) = \Omega_g.$$
We conclude that modulo $L$ we have $A^t \Omega_{g} + \Omega_{g} A=0$, as desired.

Next, we prove that $\phi$ is a homomorphism.  Consider $X,Y \in \Sp_{2g}(\Z,L)$ with
$X = \Ident_{2g} + L A$ and $Y = \Ident_{2g} + L B$.  Thus $X Y = \Ident_{2g} + L(A+B) + L^2 A B$, so
modulo $L$ we have $\phi(XY) = A+B$, as desired.

The fact that $\phi$ is surjective is a fun exercise.

Observe now that $\Ker(\phi) = \Sp_{2g}(\Z,L^2)$.  Since $\SpLie_{2g}(L)$ is abelian, this implies
that $[\Sp_{2g}(\Z,L),\Sp_{2g}(\Z,L)] < \Sp_{2g}(\Z,L^2)$.  Lemma \ref{lemma:spcommutatorcalc} then allows
us to conclude that $\Ker(\phi)=\Sp_{2g}(\Z,L^2) = [\Sp_{2g}(\Z,L),\Sp_{2g}(\Z,L)]$, and the theorem follows.
\end{proof}

\section{The Torelli group}
\label{section:torelli}

We now review some facts about $\Torelli_{g,n}$.

\begin{definition}
Let $n \in \{0,1\}$.  A {\em bounding pair} on $\Sigma_{g,n}$ is a pair $\{x_1,x_2\}$ of
disjoint nonhomotopic nonseparating curves on $\Sigma_{g,n}$ so that $x_1 \cup x_2$ separates
$\Sigma_{g,n}$.  Letting $T_{\gamma}$ denote the Dehn twist about a simple closed curve
$\gamma$, the {\em bounding pair map} associated to a bounding pair $\{x_1,x_2\}$ is
$T_{x_1} T_{x_2}^{-1}$.
\end{definition}

\noindent
Observe that if $\{x_1,x_2\}$ is a bounding pair, then $T_{x_1} T_{x_2}^{-1} \in \Torelli_{g,n}$.  Building 
on work of Birman \cite{BirmanSymplectic} and Powell \cite{PowellTorelli}, Johnson
proved the following.

\begin{theorem}[{Johnson, \cite{JohnsonFirst}}]
\label{theorem:torelligenerators}
For $g \geq 3$ and $n \in \{0,1\}$, the group $\Torelli_{g,n}$ is generated by bounding pair maps.
\end{theorem}

\begin{remark}
In fact, under the hypotheses of this theorem Johnson later proved that finitely many bounding
pair maps suffice \cite{JohnsonFinite}.  This should be contrasted with work of McCullough-Miller
\cite{McCulloughMillerTorelli} that says that for $n \in \{0,1\}$, the group $\Torelli_{2,n}$ is {\em not} finitely generated.
\end{remark}

We will also need Johnson's computation of the abelianization of $\Torelli_{g,n}$.

\begin{theorem}[{Johnson, \cite{JohnsonAbel}}]
\label{theorem:h1torelli}
Let $g \geq 3$, and set $H = \HH_1(\Sigma_g;\Z) \cong \HH_1(\Sigma_{g,1};\Z)$.  Then
$$\HH_1(\Torelli_{g,1};\Z) \cong \wedge^3 H \oplus (\text{$2$-torsion})$$
and
$$\HH_1(\Torelli_g;\Z) \cong ((\wedge^3 H) / H) \oplus (\text{$2$-torsion}).$$
\end{theorem}

\noindent
The maps
$$\tau_{g,1} : \Torelli_{g,1} \longrightarrow \HH_1(\Torelli_{g,1};\Z) / (\text{$2$-torsion}) \cong \wedge^3 H$$
and
$$\tau_{g} : \Torelli_{g} \longrightarrow \HH_1(\Torelli_{g};\Z) / (\text{$2$-torsion}) \cong (\wedge^3 H) / H$$
are known as the {\em Johnson homomorphisms} and have many remarkable properties.  For a survey,
see \cite{JohnsonSurvey}.

\section{The abelianization of $\boldsymbol{\Mod_{g,n}(L)}$}
\label{section:mainh1level}

Partly to establish notation, we begin by recalling the statement of the 5-term exact sequence in group homology.

\begin{theorem}[{see, e.g., \cite[Corollary VII.6.4]{BrownCohomology}}]
\label{theorem:fiveterm}
Let
$$1 \longrightarrow K \longrightarrow G \longrightarrow Q \longrightarrow 1$$
be a short exact sequence of groups and let $R$ be a ring.  There is then
an exact sequence
$$\HH_2(G;R) \longrightarrow \HH_2(Q;R) \longrightarrow \HH_1(K;R)_{Q} \longrightarrow \HH_1(G;R) \longrightarrow \HH_1(Q;R) \longrightarrow 0,$$
where $\HH_1(K;R)_{Q}$ is the {\em ring of co-invariants} of $\HH_1(K;R)$ under the natural action of $Q$, that is, the quotient of 
$\HH_1(K;R)$ by the ideal generated by $\{\text{$q(k)-k$ $|$ $q \in Q$ and $k \in K$}\}$.
\end{theorem}

\noindent
We will need a special case of a theorem of Broaddus-Farb-Putman that gives ``relative'' versions
of the Johnson homomorphisms on certain ``homologically defined'' subgroups of $\Mod_{g,b}$.
In our situation, the result can be stated as follows.

\begin{theorem}[{Broaddus-Farb-Putman, \cite[Example 5.3 and Theorem 5.8]{BroaddusFarbPutmanDistortion}}]
\label{theorem:reljohnson}
Fix $L \geq 2$, $g \geq 3$, and $n \in \{0,1\}$.  Set $H = \HH_1(\Sigma_{g,n};\Z)$ and
$H(L) = \HH_1(\Sigma_{g,n};\Z / L \Z)$, and define $X$ and $X(L)$ to equal $H$ and $H(L)$ if $n=0$ and
to equal $0$ if $n=1$.  Hence $(\wedge^3 H) / X$ is the target for the Johnson homomorphism
on $\Torelli_{g,n}$.  Then
there exist homomorphisms $\tau_{g,n}(L) : \Mod_{g,1}(L) \rightarrow (\wedge^3 H(L)) / X(L)$ 
that fit into the commutative diagram
$$\begin{CD}
\Torelli_{g,n} @>{\tau_{g,n}}>> (\wedge^3 H)/X\\
@VVV                      @VVV \\
\Mod_{g,n}(L) @>{\tau_{g,n}(L)}>> (\wedge^3 H(L))/X(L)
\end{CD}$$
Here the right hand vertical arrow is reduction mod $L$.
\end{theorem}

\noindent
We preface the proof of Theorem \ref{theorem:mainh1} with two lemmas.
Our first lemma was originally proven by McCarthy \cite[proof of Theorem 1.1]{McCarthyLevel}.  We
give an alternate proof.  If $G$ is a
group and $g \in G$, then denote by $[g]$ the corresponding element of $\HH_1(G;\Z)$.

\Figure{figure:crossedlantern}{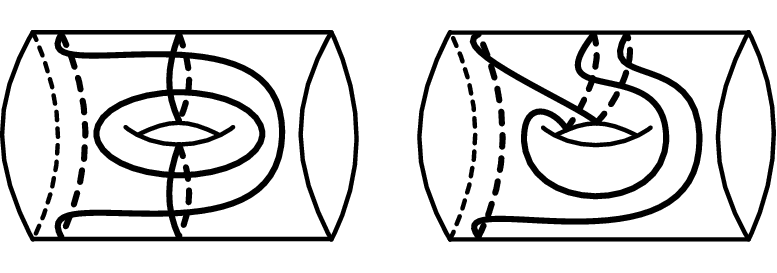}{The crossed lantern relation
$(T_{y_1}T_{y_2}^{-1})(T_{x_1}T_{x_2}^{-1})=(T_{z_1}T_{z_2}^{-1})$}

\begin{lemma}
\label{lemma:kill2torsion}
For $n \in \{0,1\}$, let $\{x_1,x_2\}$ be a bounding pair on $\Sigma_{g,n}$.
Then $L [T_{x_1} T_{x_2}^{-1}] = 0$ in $\HH_1(\Mod_{g,n}(L);\Z)$.
\end{lemma}
\begin{proof}
Embed $\{x_1,x_2\}$ in a 2-holed torus as in Figure \ref{figure:crossedlantern}.  We will make use
of the {\em crossed lantern relation} from \cite{PutmanInfinite}.  Letting
$\{y_1,y_2\}$ and $\{z_1,z_2\}$ be the other bounding pair maps depicted in Figure
\ref{figure:crossedlantern}, this relation says that
$$(T_{y_1}T_{y_2}^{-1})(T_{x_1}T_{x_2}^{-1})=(T_{z_1}T_{z_2}^{-1}).$$
Observe that for $i=1,2$ we have $z_i = T_{x_2}(y_i)$.  The key observation is that for all $n \geq 0$ we have another
crossed lantern relation
$$(T_{T_{x_2}^n(y_1)} T_{T_{x_2}^n(y_2)}^{-1}) (T_{x_1}T_{x_2}^{-1}) = (T_{T_{x_2}^{n+1}(y_1)} T_{T_{x_2}^{n+1}(y_2)}^{-1}).$$
Since $T_{x_2}^L \in \Mod_{g,n}(L)$, we conclude that in $\HH_1(\Mod_{g,n}(L);\Z)$ we have
\begin{align*}
[T_{y_1} T_{y_2}^{-1}] &= [T_{x_2}^L] + [T_{y_1} T_{y_2}^{-1}] - [T_{x_2}^{L}] = [T_{x_2}^L (T_{y_1} T_{y_2}^{-1}) T_{x_2}^{-L}] =[(T_{T_{x_2}^L(y_1)} T_{T_{x_2}^L(y_2)}^{-1})]\\
                       &= [T_{x_1}T_{x_2}^{-1}] + [(T_{T_{x_2}^{L-1}(y_1)} T_{T_{x_2}^{L-1}(y_2)}^{-1})]\\
                       &= 2[T_{x_1}T_{x_2}^{-1}] + [(T_{T_{x_2}^{L-2}(y_1)} T_{T_{x_2}^{L-2}(y_2)}^{-1})]\\
                       &\hspace{5.5pt}\vdots\\
                       &= L[T_{x_1}T_{x_2}^{-1}] + [T_{y_1} T_{y_2}^{-1}],
\end{align*}
so $L[T_{x_1}T_{x_2}^{-1}] = 0$, as desired.
\end{proof}

\noindent
For the statement of the following lemma, recall that if a group $G$ acts on a ring $R$, then
the coinvariants of that action are denoted $R_G$.

\begin{lemma}
\label{lemma:splinv}
For $L \geq 2$, define $H = \HH_1(\Sigma_{g};\Z)$ and $H(L) = \HH_1(\Sigma_{g};\Z / L\Z)$.  Then
$$(\wedge^3 H)_{\Sp_{2g}(\Z,L)} \cong \wedge^3 H(L)$$
and
$$((\wedge^3 H)/H)_{\Sp_{2g}(\Z,L)} \cong (\wedge^3 H(L))/H(L).$$
\end{lemma}
\begin{proof}
Letting $S=\{a_1,b_1,\ldots,a_g,b_g\}$ be a symplectic basis for $H$, the groups $\wedge^3 H$ and
$(\wedge^3 H) / H$
are generated by $T:=\{\text{$x \wedge y \wedge z$ $|$ $x,y,z \in S$ distinct}\}$.  Consider $x \wedge y \wedge z \in T$.
It is enough to show that in the indicated rings of coinvariants we have $L(x \wedge y \wedge z)=0$.
Now, one of $x$, $y$, and $z$ must have algebraic intersection number $0$ with the other two terms.  Assume
that $x = a_1$ and $y,z \in \{a_2,b_2,\ldots,a_g,b_g\}$ (the other cases are similar).  There is then some
$\phi \in \Sp_{2g}(\Z,L)$ so that $\phi(b_1) = b_1 + L a_1 = b_1 + L x$ and so that $\phi(y)=y$ and $\phi(z)=z$.  We
conclude that in the indicated ring of coinvariants we have $b_1 \wedge y \wedge z = (b_1 + L x) \wedge y \wedge z$, so
$L(x \wedge y \wedge z) = 0$, as desired.
\end{proof}

\begin{remark}
Lemma \ref{lemma:splinv} would {\em not} be true if $\wedge^3 H$ were replaced by $\wedge^2 H$, as
$\wedge^2 H$ contains a copy of the trivial representation of $\Sp_{2g}(\Z)$.
\end{remark}

\begin{proof}[{Proof of Theorem \ref{theorem:mainh1}}]
We will do the proof for $\Mod_{g,1}(L)$; the other case is similar.  
Let $H$ and $H(L)$ be as in Theorem \ref{theorem:reljohnson}.  Associated to the short exact sequence
$$1 \longrightarrow \Torelli_{g,1} \longrightarrow \Mod_{g,1} \longrightarrow \Sp_{2g}(\Z,L) \longrightarrow 1$$
is the 5-term exact sequence in homology given by Theorem \ref{theorem:fiveterm}.  Theorem \ref{theorem:h1torelli} says that
$$\HH_1(\Torelli_{g,1};\Z) \cong \wedge^3 H \oplus (\text{$2$-torsion})$$
and Theorem \ref{theorem:h1spl} says that $\HH_1(\Sp_{2g}(\Z,L);\Z) \cong \SpLie_{2g}(\Z / L \Z)$.  The
last 3 terms of our 5-term exact sequence are thus
$$(\wedge^3 H \oplus (\text{$2$-torsion}))_{\Sp_{2g}(\Z,L)} \stackrel{i}{\longrightarrow} \HH_1(\Mod_{g,1}(L);\Z)
  \longrightarrow \SpLie_{2g}(\Z / L \Z) \longrightarrow 0.$$
Since $L$ is odd, Lemma \ref{lemma:kill2torsion} together with Theorem \ref{theorem:torelligenerators} say
that if 
$$x \in (\wedge^3 H \oplus (\text{$2$-torsion}))_{\Sp_{2g}(\Z,L)}$$
is $2$-torsion then $i(x)=0$.  Moreover, Lemma \ref{lemma:splinv} says that
$$(\wedge^3 H)_{\Sp_{2g}(\Z,L)} \cong \wedge^3 H(L).$$
We thus obtain an exact sequence 
$$\wedge^3 H(L) \stackrel{j}{\longrightarrow} \HH_1(\Mod_{g,1}(L);\Z)
  \longrightarrow \SpLie_{2g}(\Z / L \Z) \longrightarrow 0.$$
Theorem \ref{theorem:reljohnson} then implies that $j$ is an injection, and the proof is complete.
\end{proof}

\noindent
Department of Mathematics; MIT, 2-306 \\
77 Massachusetts Avenue \\
Cambridge, MA 02139-4307 \\
E-mail: {\tt andyp@math.mit.edu}
\medskip

\end{document}